\documentclass[a4paper, 11pt]{article}
\usepackage{amsmath,amssymb,esint,amscd,xspace,fancyhdr,color,authblk,srcltx,fontenc,bbm}
\newtheorem{theorem}{Theorem}[section]

\newtheorem{corollary}[theorem]{Corollary}

\newtheorem{lemma}[theorem]{Lemma}

\newenvironment{proof}[1][Proof]{\textbf{#1.} }{\hfill\rule{0.5em}{0.5em}}
{\catcode`\@=11\global\let\AddToReset=\@addtoreset
\AddToReset{equation}{section}

\AddToReset{theorem}{section}

\title{A remark on the interpolation inequality between Sobolev spaces and Morrey spaces}
\author{Minh-Phuong Tran$^a$, Thanh-Nhan Nguyen$^b$ \\ {\small $^a$Applied Analysis Research Group, Faculty of Mathematics and Statistics,\\ Ton Duc Thang University, Ho Chi Minh City, Viet Nam\medskip\\ Email: \texttt{tranminhphuong@tdtu.edu.vn}  \medskip \medskip \\ $^b$Department of Mathematics,  Ho Chi Minh City University of Education,\\ Ho Chi Minh City, Viet Nam\medskip\\ Email: \texttt{nguyenthnhan@hcmup.edu.vn}} }
\date{\today} 		
		
\begin{document}
			
\maketitle
\begin{abstract}
Interpolation inequalities play an important role in the study of PDEs and their applications. There are still some interesting open questions and problems that related to integral estimates and regularity of solutions to the elliptic and/or parabolic equations. Since many of researchers who are working on this domain, the main motivation of our work here is to provide an important observation of $L^p$-boundedness property with respect to the interpolation inequalities. In this paper, from an idea of integral approximation theory and Sobolev, Morrey embeddings, we construct a nontrivial counterexample for interpolation inequalities in the connection of results between Sobolev and Morrey spaces. Our proofs rely on the integral representation and the theory of maximal and sharp maximal functions.

\medskip

\medskip

\medskip

\noindent 

\medskip

\end{abstract}   

\section{Introduction}
\label{sec:intro}

In the theory of partial differential equations, the study of interpolation inequalities in Lebesgue spaces, Sobolev spaces, Morrey spaces play an important role. There has been still a lot to develop in our investigations relating to these interesting inequalities, arising from $L^p$-estimates or regularity results for $L^p$ solutions of the PDEs. It is mentioned here that recently, there is an increasing literature devoted to the study of these inequalities and their improvements, such as \cite{MRR2013,VS2014, Dao2018} and references given there. In this paper, we are interested in the interpolation inequality between the Sobolev and Morrey spaces and we highlight in the proof of a boundedness property in Lebesgue spaces. More precisely and specifically, in the application of interpolation between Sobolev and Morrey spaces, we provide a counterexample, in which the boundedness property does not hold. This gives an important result that helpful in studying bounds on solution of general classes of PDEs. 

Let us consider in the Euclidean space $\mathbb{R}^d$, $d \ge 2$ and with some given parameters $1\leq p<d$ and $1<q<\frac{pd}{d-p}$, we therein focus on a brief proof of a problem related to the $L^r$ boundedness property of a function $u \in C_c^1(\mathbb{R}^d)$:
\begin{align}\label{eq:m}
\int_{\mathbb{R}^d}{|u|^r d\mathbf{x}} \leq C, \quad \forall r\in \left[p\left(\frac{q}{d}+1\right),\frac{pd}{d-p}\right].
\end{align}

This study is furthermore interconnected to solution regularity of PDEs, especially non-linear equations, and play a significant role in analysis. In particular, it makes sense to study solution properties in this range of parameter $r$. As we will show in the present work, it gives a counterexample that in \eqref{eq:m}, the inequality does not hold for some range of $r$.\

In order to state our result, we firstly recall some notations and definitions that related to our proof. Here we only recall the essentials that we use in the statements and proofs. The notations used throughout the paper are standard and some conventions are given as following. For definitions and properties of spaces and some operators we recall as below, we refer the reader to many textbooks and reference materials in \cite{Adams1,Adams2,AF,HB}.

\begin{enumerate}
\item \emph{Open ball:} If $\mathbf{x} \in \mathbb{R}^d$ and $r$ is a positive real number, we denote $B_r(\mathbf{x}) = \{\mathbf{y} \in \mathbb{R}^d: |\mathbf{y}-\mathbf{x}|<r\}$ the open ball in $\mathbb{R}^d$.
\item \emph{Average integral:} the denotation $\displaystyle{\fint_{B_r(\mathbf{x})}{f(\mathbf{y})d\mathbf{y}}}$ indicates the integral average of $f$ in the variable $y$ over the ball $B_r(\mathbf{x})$ or mean value of the function $f$ over that ball, i.e.
\begin{align*}
\fint_{B_r(\mathbf{x})}{f(\mathbf{y})d\mathbf{y}} = \frac{1}{|B_\rho(\mathbf{x})|}\int_{B_r(\mathbf{x})}{f(\mathbf{y})d\mathbf{y}},
\end{align*}
where $|B|$ denotes the $d$-dimensional Lebesgue measure of a set $B \subset \mathbb{R}^d$.
\item For any set $I \subset \mathbb{R}$, let us denote by $C^k_c(\mathbb{R}^{m},I)$ (and $C^\infty_c(\mathbb{R}^{m},I)$)  the set of $k$-th order differentiable (infinitely differentiable, respectively) function $f$  with compact support, i.e., there exists a compact set $K \subset I$ such that 
 $$\mbox{supp}(f) = \left\{\mathbf{x} \in \mathbb{R}^m: \ f(\mathbf{x}) \neq 0\right\} \subset K.$$
\item \emph{The Hardy-Littlewood maximal function:} For $\mathbf{x} \in \mathbb{R}^d$, the Hardy-Littlewood maximal function is defined for each locally integrable function $f$ in $\mathbb{R}^d$ by
\begin{align*}
\mathbf{M}(f)(\mathbf{x})=\sup_{\rho>0}\fint_{B_{\rho}(\mathbf{x})}{|f|d\mathbf{y}}.
\end{align*}
\item \emph{Sharp maximal function:} For $\mathbf{x} \in \mathbb{R}^d$, the sharp maximal function of $f$ is defined as
\begin{align}\label{eq:sharpM}
\mathbf{M}_{\#}(f)(\mathbf{x})=\sup_{\rho>0}\fint_{B_{\rho}(\mathbf{x})}{\left| f(\mathbf{y})-\fint_{B_{\rho}(\mathbf{x})}f(\mathbf{z})d\mathbf{z}\right|d\mathbf{y}}.
\end{align}
\item By $a \lesssim b$, we mean that $a \le Cb$ with some positive real constant $C$ depends
on inessential parameters.
\item We write $a \sim b$ means that there exist some positive real constants $C_1, C_2>0$ such that: $C_1 a \le b \le C_2 a$.
\end{enumerate}

The rest of this paper is structured as follows. In the next section, we state main result and some preparatory lemmas are also given therein. In Section \ref{sec:proof}, we give the detailed proof for the main theorem and some necessary proofs are also brought to reach our conclusion.

\section{Statement of main results}
\label{sec:main}
In this section, we list all the lemmas, theorems, and corollaries that will be discussed and proved. The first integral inequality is stated in Theorem \ref{lem:1} as follows, where its proof can be found in Section \ref{sec:proof}. 

%the inequality of this type is also known as Gagliardo-Nirenberg inequality, and

\begin{lemma} 
\label{lem:1}	
Let $1 < p<d$ and $1<q<\frac{pd}{d-p}$. Then, the following integral inequality holds
	\begin{align}
	\label{es}
			\int_{\mathbb{R}^d}|u|^{p\left(\frac{q}{d}+1\right)}d\mathbf{x}\leq C  \int_{\mathbb{R}^d}|\nabla u|^pd\mathbf{x} \left(\sup_{B_\rho(\mathbf{z})}\rho^{\frac{d}{q}}\fint_{B_\rho(\mathbf{z})}|u|d\mathbf{x}\right)^{\frac{qp}{d}},~~\forall~u\in C^1_c(\mathbb{R}^d).
	\end{align}
\end{lemma}

The inequality we will state in next Theorem is also known as interpolation inequality in Lebesgue spaces. As an application of the former result in \eqref{es}, Sobolev's inequality, H\"older's inequality (see \cite{HB}), we then obtain  the following simple result.

\begin{theorem} \label{theo:coro}	
Let $1< p<d$ and $1<q<\frac{pd}{d-p}$. If 
\begin{align}\label{eq:cor1}
\int_{\mathbb{R}^d}{|\nabla u|^pd\mathbf{x}} \leq 1,
\end{align}
and
\begin{align}\label{eq:cor2}
\sup_{B_\rho(\mathbf{z})}\rho^{\frac{d(q_1-q)}{q}}\int_{B_\rho(\mathbf{z})}{|u|^{q_1}d\mathbf{x}}\leq  1,
\end{align}
for some  $1\leq q_1 \le q$, then
\begin{align}
\label{1}
\int_{\mathbb{R}^d}{|u|^r d\mathbf{x}} \leq C, \quad \forall r\in \left[p\left(\frac{q}{d}+1\right),\frac{pd}{d-p}\right].
\end{align}
\end{theorem}

The proof of Theorem \ref{theo:coro} that we will show in next section \ref{sec:proof} brings us up  an interesting problem to consider. More precisely, a further question that arises here pertains to the range of $r$ in \eqref{1}, is that possible to extend value range of $r$ where the interpolation inequality \eqref{1} still holds? Dealing with this problem enables us to have a complete picture in studying of  interpolation inequalities in Lebesgue spaces, and somewhat related to Gagliardo-Nirenberg inequalities also their generalizations (we refer the reader to \cite{Nirenberg1959,Adams1,Adams2,MRR2013,Dao2018} and literature related to this subject). This question follows our interest in the theory of PDEs, that could provide a procedure to construct regularity, some comparison estimates or important properties of solutions to elliptic and/or parabolic equations in future studies. That is the reason why in the present work, our study on this problem could be considered.

The following result of next Corollary \ref{rmk} asserts that inequality \eqref{1} holds for extended range of $r$ only for a specific case $q_1=q$. Also, we refer the reader here to Section \ref{sec:proof} for a brief proof.
\begin{corollary}
\label{rmk} 
If $q_1=q$ in Corollary \ref{theo:coro}, then \eqref{1} is still true for any $r\in \left[q,p\left(\frac{q}{d}+1\right)\right].$
\end{corollary}

However, as we will show in this paper, this interpolation inequality does not hold for $r < p\left(\frac{q}{d}+1\right)$. It provides us a very important result, that the range of $r$ in Theorem \ref{theo:coro} is optimal. Indeed, a counterexample is a specific instance where the inequality \eqref{1} does not hold outside the interval $\left[q,p\left(\frac{q}{d}+1\right)\right]$, is stated in the following theorem.

%and the assertion  of Lemma \ref{lem:1} is extremely useful tool for proving the counterexample we are interested in (see Theorem \ref{theo:main} that we will show later in this section)

\begin{theorem} 
\label{theo:main}
Let $1< p<d$ and $1<q<\frac{pd}{d-p}$. Then, for any $1\leq q_1<q$, there exists a sequence $(u_n)_n \in C^\infty_c(\mathbb{R}^d)$ such that 
\begin{align}
\label{es5}
\int_{\mathbb{R}^d}|\nabla u_n|^pd\mathbf{x}\leq 1,~~ \sup_{B_\rho(\mathbf{z})}\rho^{\frac{d(q_1-q)}{q}}\int_{B_\rho(\mathbf{z})}{|u_n|^{q_1}d\mathbf{x}} \leq  1.
\end{align}
Moreover, for any $r>0$ there holds
\begin{align}
\label{es6}
\int_{\mathbb{R}^d}{|u_n|^r d\mathbf{x}} \sim 2^{\frac{d^2(q-q_1)}{q(dp-(d-p)q)}\left(p(1+\frac{q}{d})-r\right)n}, \quad \forall n \gg 1.
\end{align}
In particular, 
	\begin{align}\label{es7}
	\int_{\mathbb{R}^d}{|u_n|^r d\mathbf{x}} \leq C, \quad \forall n\gg 1
	\end{align}
if and only if	$r\geq p\left(\frac{q}{d}+1\right)$.
\end{theorem}

\section{Proof of main results}
\label{sec:proof}

This section is dedicated to proofs of our statements in previous section \ref{sec:main}. We note that the positive constant $C$ may change from an inequality to another in the proofs. \\

\begin{proof}[Proof of Lemma~\ref{lem:1}] 

By Poincar\'e inequality, for $\rho>0$ and any $\mathbf{y} \in \mathbb{R}^d$, we firstly have
\begin{align} \label{eq:1}
	\fint_{B_{\rho}(\mathbf{y})}|u-(u)_{B_{\rho}(\mathbf{y})}|d\mathbf{x}\leq C \rho \fint_{B_{\rho}(\mathbf{y})}|\nabla u|d\mathbf{x}~~,\quad \forall u\in C^1_c(\mathbb{R}^d),
\end{align}
where the term $(f)_{\Omega}$ stands for the average value of $f$ over the domain $\Omega$.
%Let $u \in C_c^1(\mathbb{R}^d)$ and $\rho>0$ it writes
%\begin{align}
%\label{eq:1}
%	\rho^{-1} \fint_{B_{\rho}(\mathbf{y})}|u-(u)_{B_{\rho}(\mathbf{y})}|d\mathbf{x}\leq \fint_{B_{\rho}(\mathbf{y})}|\nabla u|d\mathbf{x}.
%\end{align}
Taking the supremum both sides of \eqref{eq:1} for all $\rho>0$, this gives
\begin{align}\label{est:1b}
\sup_{\rho>0}\rho^{-1}\fint_{B_\rho(\mathbf{y})}|u-(u)_{B_\rho(\mathbf{y})}|d\mathbf{x}\leq C \mathbf{M}(|\nabla u|)(\mathbf{y}),\quad \forall \mathbf{y}\in \mathbb{R}^d.
\end{align}

By the definition of the sharp maximal function $\mathbf{M}_{\#}$ in \eqref{eq:sharpM}, it enables us to decompose it as
\begin{align}
\nonumber
& \left[\mathbf{M}_{\#}(u)(\mathbf{y})\right]^{p\left(\frac{q}{d}+1\right)} \\ \nonumber
& =   \sup_{\rho>0} \left(\rho^{\frac{d}{q}} \fint_{B_\rho(\mathbf{y})}|u-(u)_{B_\rho(\mathbf{y})}|d\mathbf{x}\right)^{\frac{pq}{d}} . \left(\rho^{-1}\fint_{B_\rho(\mathbf{y})}|u-(u)_{B_\rho(\mathbf{y})}|d\mathbf{x}\right)^{p} \\ \label{est:1c}
& \le {C \left(\sup_{\rho>0,\, \mathbf{z}\in \mathbb{R}^n}\rho^{\frac{d}{q}}\fint_{B_\rho(\mathbf{z})}{|u|d\mathbf{x}}\right)^{\frac{qp}{d}}} . \left(\sup_{\rho>0} \rho^{-1}\fint_{B_\rho(\mathbf{y})}|u-(u)_{B_\rho(\mathbf{y})}|d\mathbf{x}\right)^{p}.  
\end{align}

Thanks to~\eqref{est:1b} and~\eqref{est:1c}, one also has the related inequality between the Hardy-Littlewood maximal function $\mathbf{M}$ and the sharp maximal function $\mathbf{M}_{\#}$ as in the following formula:
\begin{align}
\label{es3}
\left[\mathbf{M}_{\#}(u)(\mathbf{y})\right]^{p\left(\frac{q}{d}+1\right)}\leq C \left(\sup_{B_\rho(\mathbf{z})}\rho^{\frac{d}{q}}\fint_{B_\rho(\mathbf{z})}{|u|d\mathbf{x}}\right)^{\frac{qp}{d}} \left[\mathbf{M}(|\nabla u|)(\mathbf{y})\right]^{p}.  
\end{align}
%where $\displaystyle{\sup_{B_\rho(\mathbf{z})}}$ denotes $\displaystyle{\sup_{\rho>0,\, \mathbf{z}\in \mathbb{R}^n}}$. 

Refer to \cite{EMS}, it is well known that for any $f\in L^s(\mathbb{R}^d)$ and $s \in (1,\infty)$, there exists a positive constant $C$ such that the following inequalities hold:
\begin{align}
\label{es4}
&C^{-1}	\int_{\mathbb{R}^d}{|f|^s d\mathbf{x}} \leq 	\int_{\mathbb{R}^d}{|\mathbf{M}_{\#}(f)|^s d\mathbf{x}} \leq C 	\int_{\mathbb{R}^d}{|f|^s d\mathbf{x}},\\&\label{es41}
C^{-1}	\int_{\mathbb{R}^d}{|f|^s d\mathbf{x}} \leq 	\int_{\mathbb{R}^d}{|\mathbf{M}(f)|^s d\mathbf{x}}\leq C 	\int_{\mathbb{R}^d}{|f|^s d\mathbf{x}}.
\end{align}

In what follows from \eqref{es3}, \eqref{es4} and \eqref{es41} that 
\begin{align*}
\int_{\mathbb{R}^d}|{u|^{p\left(\frac{q}{d}+1\right)}d\mathbf{x}}&\leq C \int_{\mathbb{R}^d} {\left[\mathbf{M}_{\#}(u)(\mathbf{x})\right]^{p\left(\frac{q}{d}+1\right)}d\mathbf{x}}\\
&\leq C \int_{\mathbb{R}^d}\left[\mathbf{M}(|\nabla u|)(\mathbf{x})\right]^p \left(\sup_{B_\rho(\mathbf{z})}\rho^{\frac{d}{q}}\fint_{B_\rho(\mathbf{z})}{|u(\mathbf{y})|d\mathbf{y}}\right)^{\frac{qp}{d}}d\mathbf{x}\\
&\leq C  \int_{\mathbb{R}^d}|\nabla u|^p \left(\sup_{B_\rho(\mathbf{z})}\rho^{\frac{d}{q}}\fint_{B_\rho(\mathbf{z})}{|u(\mathbf{y})|d\mathbf{y}}\right)^{\frac{qp}{d}}d\mathbf{x}\\
& \leq C \left(\sup_{B_\rho(\mathbf{z})}\rho^{\frac{d}{q}}\fint_{B_\rho(\mathbf{z})}{|u|d\mathbf{x}}\right)^{\frac{qp}{d}} \int_{\mathbb{R}^d}|\nabla u|^p d\mathbf{x},
\end{align*}
and therefore, our proof is complete. 
\end{proof}

\begin{proof}[Proof of Theorem~\ref{theo:coro}]

The strengthened form of Sobolev's inequality asserts that for all $u \in W^{1,p}(\mathbb{R}^d)$, one has
\begin{align*}
\left(\int_{\mathbb{R}^d}{|u|^{\frac{pd}{d-p}}d\mathbf{x}}\right)^{\frac{d-p}{dp}}\leq C \left( \int_{\mathbb{R}^d}{|\nabla u|^p d\mathbf{x}}\right)^{\frac{1}{p}},
\end{align*}
which deduces that if~\eqref{eq:cor1} holds, then 
\begin{align} \label{est:lem1}
\int_{\mathbb{R}^d}{|u|^{\frac{pd}{d-p}}d\mathbf{x}}\leq C.
\end{align}
Thanks to H\"older's inequality, for all $u$ satisfying~\eqref{eq:cor2}, there holds
\begin{align*}
\sup_{B_\rho(\mathbf{z})}\rho^{\frac{d}{q}}\fint_{B_\rho(\mathbf{z})}{|u|d\mathbf{x}}
& = \sup_{B_\rho(\mathbf{z})}{\rho^{\frac{d(1-q)}{q}}} \int_{B_\rho(\mathbf{z})}{|u|d\mathbf{x}}  \\
& \leq \sup_{B_\rho(\mathbf{z})}{\rho^{\frac{d(1-q)}{q}}} . \ \rho^{\frac{d(q_1-1)}{q_1}}\left(\int_{B_\rho(\mathbf{z})}{|u|^{q_1}d\mathbf{x}}\right)^{\frac{1}{q_1}}  \\
& \leq \left(\sup_{B_\rho(\mathbf{z})}{\rho^{\frac{d(q_1-q)}{q}}}\int_{B_\rho(\mathbf{z})}{|u|^{q_1}d\mathbf{x}} \right)^{\frac{1}{q_1}}\\
& \le 1,
\end{align*}
for any $1 \leq q_1 \le q$. In the use of Lemma~\ref{lem:1}, one obtains that
\begin{align} \label{est:lem2}
\int_{\mathbb{R}^d}{|u|^{p\left(\frac{q}{d}+1\right)}d\mathbf{x}}\leq C.
\end{align}
%The estimate~\eqref{1} holds by combining~\eqref{est:lem1}-\eqref{est:lem2} and the interpolation inequality in Lebesgue space which is presented below for the convenience of the reader. \\
For all $ r \in \left[p\left(\frac{q}{d}+1\right),\frac{pd}{d-p}\right]$, it is easily seen that there exists $\theta \in [0,1]$ such that
\begin{align*}
\frac{1}{r} = \frac{\theta d}{p(q+d)} + \frac{(1-\theta)(d-p)}{pd}.
\end{align*}
Applying H\"older's inequality, one gets that
\begin{align}\label{eq:interineq}
\begin{split}
\int_{\mathbb{R}^d} |u|^rd \mathbf{x} & =  \int_{\mathbb{R}^d} |u|^{r\theta} |u|^{r(1-\theta)}d \mathbf{x} \\
& \le \left(\int_{\mathbb{R}^d} \left[|u|^{r\theta}\right]^{\frac{p\left(\frac{q}{d}+1\right)}{r\theta}}d \mathbf{x}\right)^{\frac{r\theta d}{p(q+d)}} \left(\int_{\mathbb{R}^d} \left[|u|^{r(1-\theta)}\right]^{\frac{pd}{r(1-\theta)(d-p)}}d \mathbf{x}\right)^{\frac{r(1-\theta)(d-p)}{pd}}\\
& \le \left(\int_{\mathbb{R}^d} |u|^{p\left(\frac{q}{d}+1\right)}d \mathbf{x}\right)^{\frac{r\theta d}{p(q+d)}} \left(\int_{\mathbb{R}^d} |u|^{\frac{pd}{d-p}}d \mathbf{x}\right)^{\frac{r(1-\theta)(d-p)}{pd}}.
\end{split}
\end{align}

It should be mentioned that the obtained inequality \eqref{eq:interineq} is also known as the interpolation inequality in Lebesgue spaces, that has been studied in \cite{EMS,Stein2011,Tao}. 

From what already been proved in ~\eqref{est:lem1},~\eqref{est:lem2} and \eqref{eq:interineq}, the proof of Theorem \ref{theo:coro} is then complete.
\end{proof}

Next, let us now give a proof of what stated in Corollary \ref{rmk}.

\begin{proof}[Proof of Corollary~\ref{rmk}]

If $q_1 = q$ in Theorem~\ref{theo:coro}, the assumption~\eqref{eq:cor2} becomes
\begin{align*}
\sup_{B_\rho(\mathbf{z})}\int_{B_\rho(\mathbf{z})}{|u|^{q}d\mathbf{x}}\leq  1,
\end{align*} 
and it follows that 
\begin{align}\label{eq:rmk}
\int_{\mathbb{R}^n}{|u|^{q}d\mathbf{x}}\leq  1.
\end{align} 
The estimate~\eqref{1} holds for any $r \in \left[q, p\left(\frac{q}{d}+1\right)\right]$ from~\eqref{est:lem2}, \eqref{eq:rmk} and repeated application of interpolation inequality in Lebesgue space as in \eqref{eq:interineq}.
\end{proof}

\begin{proof}[Proof of Theorem~\ref{theo:main}]

Let $\phi\in C^\infty_c(\mathbb{R}^{d-1},[0,1])$ and $\eta\in C_c^\infty(\mathbb{R},[0,1])$ be given as 
\begin{align*}
\phi (\mathbf{x}) = \begin{cases} 
1, \quad \mathbf{x} \in B_1(0)\\
0, \quad \mathbf{x} \in B_2(0)^c
\end{cases},
\qquad
\eta(t) = \begin{cases}
1, \quad t \in (-1,1)\\
0, \quad t \in (-2,2)^c
\end{cases}
\end{align*}
and satisfy $|\nabla \phi(\mathbf{x})| \leq 1, \ \forall \mathbf{x} \in \mathbb{R}^{d-1}$ and $|\eta'(t)|\leq 1, \ \forall t \in \mathbb{R}$.\\

Moreover, let us choose a parameter $\theta\in \left(0,10^{-10d}\right)$ and  $n\geq 100/\theta$. Then, for any fixed $k\geq 10/\theta$,  we set the sequence $\{a_{k,j}\}$ in terms of
\begin{align*}
a_{k,j}=2^{k-1}+1+j2^{k-1-\theta(k-1)}, \quad \mbox{ for all } 1\leq j\leq 2^{\theta(k-1)}-3.
\end{align*}
It is easy to check that for any $1\leq j\leq 2^{\theta(k-1)}-3$, one has 
$$a_{k,j}\in (2^{k-1}+2,2^k-2).$$ 
Let us set functions $\chi$ and $\sigma$ given as follows
\begin{align*}
\chi_n(t)=\sum_{k\geq 10\theta^{-1}}^{n}\sum_{j=1}^{2^{\theta(k-1)}-3}\eta\left(t-a_{k,j}\right) \ \quad \text{and} \ \quad \sigma_n(t)= \chi_n(t)+\chi_n(-t).
\end{align*}
Firstly, since 
$$\mbox{supp}(\eta(.-a_{k,j}))\cap \mbox{supp}(\eta(.-a_{k',j'}))=\emptyset,~~\forall (k,j)\not= (k',j'),$$  
we then have that  for any $r>0$,
\begin{align*}
\sup_{\rho>0, t_0\in \mathbb{R}}\left(\rho^{-\theta}\int_{t_0-\rho}^{t_0+\rho}\sigma_n(t)^rdt\right)~\sim~1,~~\forall n\geq 100/\theta,
\end{align*}
and 
\begin{align}\label{es1}
\int_{\mathbb{R}} \sigma_n(t)^rdt \sim 2^{\theta n},~~\forall n\geq 100/\theta,
\end{align}
together with that fact that
\begin{align}\label{es2}
\int_{\mathbb{R}} |\sigma_n'(t)|^rdt \sim 2^{\theta n},~~\forall n\geq 100.
\end{align}
On the other hand, for $n>100/\theta$, we define the following sequence of functions:
\begin{align}
u_n(t,\mathbf{x})=2^{-\frac{\alpha dn}{q} }\sigma_n\left(2^{-\alpha n}t\right)\phi\left(2^{-\alpha n}\mathbf{x}\right),\quad \forall (t,\mathbf{x})\in \mathbb{R}\times \mathbb{R}^{d-1},
\end{align}
where $\theta=\frac{d(q-q_1)}{q}\in \left(0,10^{-10d}\right)$ and $\alpha=\frac{d(q-q_1)}{dp-(d-p)q}$.\\

We now prove that $u_n$ satisfies \eqref{es5} and \eqref{es6}. Indeed, 
\begin{itemize}
\item[(i)] Firstly, one can compute 
\begin{align}\nonumber
\int_{\mathbb{R}^d}{|\nabla u_n|^p d\mathbf{x} dt} &=2^{-\frac{p\alpha dn}{q} }\int_{\mathbb{R}}|\sigma_n'(2^{-\alpha n}t)|^p2^{-p\alpha n}dt\int_{\mathbb{R}^{d-1}}\phi\left(2^{-\alpha n}\mathbf{x}\right)^p d\mathbf{x}\\ \nonumber
& \qquad +2^{-\frac{p\alpha dn}{q} }\int_{\mathbb{R}}|\sigma_n(2^{-\alpha n}t)|^pdt\int_{\mathbb{R}^{d-1}}|\nabla\phi\left(2^{-\alpha n}\mathbf{x}\right)|^p2^{-p\alpha n}d\mathbf{x}\\ \nonumber
&= 2^{-\frac{p\alpha dn}{q} }\int_{\mathbb{R}}|\sigma_n'(t)|^p2^{\alpha n-p\alpha n}dt\int_{\mathbb{R}^{d-1}}\phi\left(\mathbf{x}\right)^p2^{(d-1)\alpha n}d\mathbf{x}\\ \nonumber
&\qquad +2^{-\frac{p\alpha dn}{q} }\int_{\mathbb{R}}|\sigma_n(t)|^p2^{\alpha n}dt\int_{\mathbb{R}^{d-1}}|\nabla\phi(\mathbf{x})|^p2^{(d-1)\alpha n-p\alpha n}d\mathbf{x}\\ \nonumber
&= 2^{n\left({-\frac{p\alpha d}{q}+d\alpha -p\alpha}\right)}\left(\int_{\mathbb{R}}|\sigma_n'(t)|^pdt \int_{\mathbb{R}^{d-1}}\phi(\mathbf{x})^pd\mathbf{x} \right. \\ \label{est:t1}
 & \qquad \qquad \qquad \qquad \qquad \qquad \left. +\int_{\mathbb{R}}|\sigma_n(t)|^p dt\int_{\mathbb{R}^{d-1}}{|\nabla\phi(\mathbf{x})|^p}d\mathbf{x}\right).
\end{align}
For $\theta$ and $\alpha$ are chosen above, it is easy to obtain $-\frac{p\alpha d}{q}+d\alpha -p\alpha +\theta =0$, and by \eqref{es1}, \eqref{es2} and \eqref{est:t1} one has
\begin{equation}
\int_{\mathbb{R}^d}|\nabla u_n|^p \sim 2^{-\frac{p\alpha dn}{q}+d\alpha n-p\alpha n+\theta n}=1,~~\forall n>100/\theta.
\end{equation}
which implies the first inequality of \eqref{es5}.
\item[(ii)] By changing variables inside the integrals, one gets
\begin{align*}
&\sup_{B_\rho(s,y)\subset \mathbb{R}^d}\left(\rho^{-\theta}\int_{B_\rho(s,y)}|u_n(t,\mathbf{x})|^{q_1}dtd\mathbf{x}\right) \sim\sup_{\rho>0}\left(\rho^{-\theta}\int_{B_\rho(0)}|u_n(t,\mathbf{x})|^{q_1}dtd\mathbf{x}\right) \\
&~~~\sim 2^{-\frac{q_1\alpha dn}{q} }\sup_{\rho>0} \left(\rho^{-\theta}\int_{-\rho}^{\rho}\sigma_n\left(2^{-\alpha n}t\right)^{q_1}dt \int_{|\mathbf{x}|<\rho}\phi\left(2^{-\alpha n}\mathbf{x}\right)^{q_1}d\mathbf{x}\right)
\\
&~~~\sim 2^{-\frac{q_1\alpha dn}{q} }\sup_{\rho>0} \left(\rho^{-\theta}\int_{-2^{-\alpha n}\rho}^{2^{-\alpha n}\rho}\sigma_n\left(t\right)^{q_1}2^{\alpha n}dt \int_{|\mathbf{x}|<2^{-\alpha n}\rho}\phi\left(\mathbf{x}\right)^{q_1} 2^{(d-1)\alpha n}d\mathbf{x}\right) \\
&~~~\sim 2^{-\frac{q_1\alpha dn}{q} }\sup_{\rho> 2^{\alpha n}}\left( \rho^{-\theta}\int_{-2^{-\alpha n}\rho}^{2^{-\alpha n}\rho}\sigma_n\left(t\right)^{q_1}2^{\alpha n}dt \int_{|\mathbf{x}|<2^{-\alpha n}\rho}\phi\left(\mathbf{x}\right)^{q_1} 2^{(d-1)\alpha n}d\mathbf{x} \right)\\
&~~~\overset{\eqref{es1}}\sim 2^{-\frac{q_1\alpha dn}{q} } 2^{\alpha dn}2^{-\alpha\theta n}=1.
\end{align*}
Thus, 
\begin{align*}
\sup_{B_\rho(s,y)\subset \mathbb{R}^d}\rho^{-\theta}\int_{B_\rho(s,y)}|u_n|^{q_1}d\mathbf{x}\sim 1,~~\forall n>100/\theta.
\end{align*}
which implies the second inequality of  \eqref{es5}.
\item[(iii)] On the other hand,  for any $r>0$, it gives the following estimate
\begin{align*}
\int_{\mathbb{R}^d}|u_n(t,\mathbf{x})|^rdtd\mathbf{x}&=2^{-\frac{r\alpha dn}{q} }\int_{\mathbb{R}}\sigma_n\left(2^{-\alpha n}t\right)^rdt \int_{\mathbb{R}^{d-1}}\phi\left(2^{-\alpha n}\mathbf{x}\right)^pd\mathbf{x}\\
& =2^{-\frac{r\alpha dn}{q}+d\alpha n}\int_{\mathbb{R}}\sigma_n\left(t\right)^rdt \int_{\mathbb{R}^{d-1}}\phi\left(\mathbf{x}\right)^pd\mathbf{x}\\
&\overset{\eqref{es1}} \sim 2^{-\frac{r\alpha dn}{q}+d\alpha n+\theta n}=2^{\frac{d^2(q-q_1)}{q(dp-(d-p)q)}\left(p(1+\frac{q}{d})-r\right)n}.
\end{align*}
which implies \eqref{es6}.
\end{itemize}
Finally, from what already been proved in~\eqref{es6}, it concludes that the estimate~\eqref{es7} holds if and only if $r \ge p\left(\frac{q}{d}+1\right)$. 

\end{proof}

\end{document}